\documentclass[oneside,english]{amsart}
\usepackage[T1]{fontenc}
\usepackage[latin9]{inputenc}
\usepackage{color}
\usepackage{textcomp}
\usepackage{amstext}
\usepackage{amsthm}
\usepackage{amssymb}

\usepackage[backend=bibtex8,maxbibnames=99,style=numeric]{biblatex}
\bibliography{fracInt}

\AtEveryBibitem{\clearfield{month}
\clearfield{day}
\clearfield{doi}
\clearfield{url}
}

\renewbibmacro{in:}{}
\DeclareFieldFormat
  [article,inbook,incollection,inproceedings,patent,thesis,unpublished]
  {title}{#1\isdot}
\makeatletter
\def\url@leostyle{%
  \@ifundefined{selectfont}{\def\UrlFont{\sf}}{\def\UrlFont{\small\sffamily}}}
\makeatother
\makeatletter
\numberwithin{equation}{section}
\numberwithin{figure}{section}
\theoremstyle{plain}
\newtheorem{thm}{\protect\theoremname}
  \theoremstyle{plain}
  \newtheorem{cor}{\protect\corollaryname}
  \theoremstyle{plain}
  \newtheorem{lem}{\protect\lemmaname}
  \theoremstyle{remark}
  \newtheorem{rem}{\protect\remarkname}

\DeclareMathOperator {\esssup}{ess \ sup}
\DeclareMathOperator {\supp}{supp}
\DeclareMathOperator {\BMO}{BMO}
\usepackage{pgf,tikz}
\usetikzlibrary{arrows}
\usepackage{url}

\makeatother

\usepackage{babel}
  \providecommand{\lemmaname}{Lemma}
   \providecommand{\corollaryname}{Corollary}
  \providecommand{\remarkname}{Remark}
\providecommand{\theoremname}{Theorem}

\usepackage[
  colorlinks,
  linkcolor=red
]{hyperref}
\newcommand\rurl[2]{%
  \href{http://#1}{\nolinkurl{#1}}%
}

\begin{document}

\title[Bloom estimates for iterated commutators of fractional integrals]{On Bloom type estimates for iterated commutators of fractional integrals}

\author[N. Accomazzo]{Natalia Accomazzo} 
\address[N. Accomazzo]{ Department of Mathematics, University of the Basque Country, Bilbao, Spain}
\email{nataliaceleste.accomazzo@ehu.eus}

\author[J.C. Mart\'{\i}nez-Perales]{Javier C. Mart\'{\i}nez-Perales}
\address[J.C. Mart\'{\i}nez-Perales]{
BCAM\textendash  Basque Center for Applied Mathematics, Bilbao, Spain}
\email{jmartinez@bcamath.org}

\author[I.P. Rivera-R\'{\i}os]{Israel P. Rivera-R\'{\i}os}
\address[I.P. Rivera-R\'{\i}os]{ Department of Mathematics, University of the Basque Country and
BCAM \textendash Basque Center for Applied Mathematics, Bilbao, Spain}
\email{petnapet@gmail.com}

\begin{abstract}
In this paper we provide quantitative Bloom type estimates for iterated commutators of fractional integrals improving and extending results from \cite{HRS}. We give new proofs for those inequalities relying upon a new sparse domination that we provide as well in this paper and also in techniques developed in the recent paper \cite{LORR2}. We extend as well the necessity established in \cite{HRS} to iterated commutators providing a new proof. As a consequence of the preceding results we recover the one weight estimates in \cite{CUM, BMMST} and establish the sharpness in the iterated case. Our result provides as well a new characterization of the $\BMO$ space.
\end{abstract}

\maketitle

\section{Introduction and main results}

We recall that given $0<\alpha<n$, the fractional integral operator $I_\alpha$, or Riesz potential, on $\mathbb{R}^n$ is defined by 
\[I_\alpha f(x):=\int_{\mathbb{R}^n}\frac{f(y)}{|x-y|^{n-\alpha}}dy\]
and the corresponding associated maximal function $M_\alpha$ by
\[M_\alpha f(x)=\sup_{x\in Q}\frac{1}{|Q|^{1-\frac{\alpha}{n}}}\int_Q|f(y)|dy.\]

Those operators are bounded from $L^p$ to $L^q$ provided that $1<p<\frac{n}{\alpha}$ and $\frac{1}{q}+\frac{\alpha}{n}=\frac{1}{p}$ (see \cite[Chapter 5]{S} for those and other classical related results). 

The study of weighted estimates for these operators and slightly more general ones, is not interesting just for its own sake but also for its applications to partial di\-ffe\-ren\-tial equations, Sobolev embeddings or quantum mechanics (see for instance \cite[Section 9]{FJW} or \cite{SW}).  $A_{p,q}$ weights, which were introduced by Muckenhoupt and Wheeden \cite{MW}, can be considered the class that governs the behaviour of fractional operators. We recall that, given $1<p<q<\infty$,  $w\in A_{p,q}$ if
\[
[w]_{A_{p,q}}=\sup_{Q}\frac{1}{|Q|}\int_{Q}w^{q}\left(\frac{1}{|Q|}\int_{Q}w^{-p'}\right)^{\frac{q}{p'}}<\infty.
\]
Since $1<p<q<\infty$, using H\"older's inequality, it is not hard to check that
\begin{equation}\label{eq:ApAqApq}
[w^{p}]_{A_{p}}\leq[w]_{A_{p,q}}^{\frac{p}{q}}\quad\text{and}\quad[w^q]_{A_{q}}\leq[w]_{A_{p,q}}.
\end{equation}
where, $A_r$ ($1<r<\infty$) is the the Muckenhoupt class, namely, $v\in A_r$ if
\[[v]_{A_r}=\sup_Q\frac{1}{|Q|}\int_{Q}v\left(\frac{1}{|Q|}\int_{Q}v^{-\frac{r'}{r}}\right)^{\frac{r}{r'}}<\infty.\]
%
During the last decade, many authors have devoted plenty of works to the study of quantitative weighted estimates, in other words, estimates in  which the quantitative dependence on the $A_p$ constant $[w]_{A_p}$ or, in its case, on the $A_{p,q}$ constant $[w]_{A_{p,q}}$, is the central point. The $A_2$ Theorem, namely the linear dependence on the $A_2$ constant for Calder\'on-Zygmund operators \cite{H} can be considered the most representative result in this line. In the case of fractional integrals, 
the sharp dependence on the $A_{p,q}$ constant was obtained by Lacey, Moen, P\'erez and Torres \cite{LMPT}. The precise statement is the following
\begin{thm}\label{Thm:LMPT}
Let $\alpha\in(0,n)$ and $1<p<\frac{n}{\alpha}$
and $q$ defined by $\frac{1}{q}+\frac{\alpha}{n}=\frac{1}{p}$. Then, if $w\in A_{p,q}$ we have that
\[
\|I_\alpha f\|_{L^{q}(w^{q})}\leq c_{n,\alpha}[w]_{A_{p,q}}^{\left(1-\frac{\alpha}{n}\right)\max\left\{ 1,\frac{p'}{q}\right\} } \|f\|_{L^{p}(w^{p})}
\]
and the estimate is sharp in the sense that the inequality does not hold if we replace the exponent of the $A_{p,q}$ constant by a smaller one.
\end{thm}

Given a locally integrable function $b$ and a linear operator $G$, the commutator $[b,G]$ is the operator defined by
\[ [b,G]f(x)=b(x)Gf(x)-G(bf)(x)\]
and the iterated commutator of order $m$, $G_b^m$, is defined inductively for $b\in L^m_{\text{loc}}(\mathbb{R}^n)$, by
\[G_b^mf(x)=[b,G_b^{m-1}]f(x)\]
where $G_b^0=G$. 

Returning to quantitative estimates, the counterpart of Theorem \ref{Thm:LMPT} for commutators was obtained by Cruz-Uribe and Moen \cite{CUM}. The precise statement of their result is the following.
\begin{thm}\label{Thm:CUM}
Let $\alpha\in(0,n)$ and $1<p<\frac{n}{\alpha}$
and $q$ defined by $\frac{1}{q}+\frac{\alpha}{n}=\frac{1}{p}$. Then, if $w\in A_{p,q}$ and $b\in \BMO$ we have that
\[
\|[b,I_\alpha] f\|_{L^{q}(w^{q})}\leq c_{n,\alpha}\|b\|_{\BMO}[w]_{A_{p,q}}^{\left(2-\frac{\alpha}{n}\right)\max\left\{ 1,\frac{p'}{q}\right\} }\|f\|_{L^{p}(w^{p})}
\]
and the estimate is sharp in the sense that the inequality does not hold if we replace the exponent of the $A_{p,q}$ constant by a smaller one.
\end{thm}
One of the main purposes of this paper is to obtain quantitative two weight estimates for iterated commutators of fractional integrals assuming that the symbol $b$ belongs to a ``modified'' $\BMO$ class. 
The motivation to study this kind of estimates can be traced back to 1985 to the work of Bloom \cite{B}. For the Hilbert transform $H$, he proved that if $\mu,\lambda\in A_p$ and $\nu=\left(\frac{\mu}{\lambda}\right)^\frac{1}{p}$, then \[\|[b,H]f\|_{L^p(\lambda)}\leq c_{\mu,\lambda} \|f\|_{L^p(\mu)}\] if and only if $b\in \BMO_\nu$, namely, $b$ is a locally integrable function such that
\[\|b\|_{\BMO_\nu}=\sup_Q\frac{1}{\nu(Q)}\int_Q|b-b_Q|<\infty.\] 
Some years later, Garc\'{\i}a-Cuerva, Harboure, Segovia and Torrea \cite{GCHST}, extended the sufficiency of that result to iterated commutators of strongly singular integrals, assuming  that $b\in \BMO_{\nu^{\frac{1}{m}}}$, where $m$ stands for the order of the commutator.

Recently Lacey, Holmes and Wick  \cite{HLW} extended Bloom's result to Calder\'on-Zygmund operators, and a quantitative version of the sufficiency in that result was provided in \cite{LORR}.

For iterated commutators of Calder\'on-Zygmund operators, Holmes and Wick \cite{HW} established that $b\in \BMO\cap \BMO_\nu$ is a sufficient condition for the two weights estimate to hold. However that result was substantially improved in \cite{LORR2} where it was proved that $\BMO\cap \BMO_\nu \subset  \BMO_{\nu^{\frac{1}{m}}}$ and that $b\in \BMO_{\nu^{\frac{1}{m}}}$ is also a necessary condition, besides providing a quantitative version of the sufficiency under the same condition.

At this point we present our contribution. Combining a sparse domination result that will be presented in Section \ref{Sec:Sparse} with techniques in \cite{LORR2} we obtain the following result.

\begin{thm}
\label{thm:BloomFrac}Let $\alpha\in(0,n)$ and $1<p<\frac{n}{\alpha}$, $q$ defined by $\frac{1}{q}+\frac{\alpha}{n}=\frac{1}{p}$ and $m$ a positive integer. Assume that $\mu,\lambda\in A_{p,q}$ and that $\nu=\frac{\mu}{\lambda}$.
If $b\in \BMO_{\nu^\frac{1}{m}}$, then
\begin{equation}\label{eq:depIalphabm}
\|(I_\alpha)_b^mf\|_{L^{q}(\lambda^{q})}\leq c_{m,n,\alpha,p}\|b\|_{\BMO_{\nu^{\frac{1}{m}}}}^m \kappa_m\|f\|_{L^{p}(\mu^{p})},
\end{equation}
where 
\[
\kappa_m=\sum_{h=0}^m\binom{m}{h} \left([\lambda]_{A_{p,q}}^{\frac{h}{m}}[\mu]_{A_{p,q}}^{\frac{m-h}{m}}\right)^{\left(1-\frac{\alpha}{n}\right)\max\left\{1,\frac{p'}{q}\right\}}P(m,h)Q(m,h)\]
and
\[\begin{split}
P(m,h)&\leq \left([\lambda^q]_{A_q}^{\frac{m+(h+1)}{2}}[\mu^q]_{A_q}^{\frac{m-(h+1)}{2}}\right)^{\frac{m-h}{m}\max\{1,\frac{1}{q-1}\}}
\\
Q(m,h)&\leq \left([\lambda^p]_{A_p}^{\frac{h-1}{2}}[\mu^p]_{A_p}^{m-\frac{h-1}{2}}\right)^{\frac{h}{m}\max\{1,\frac{1}{p-1}\}}.
\end{split}
\]
Conversely if for every set $E$ of finite measure we have that
\begin{equation}\label{eq:StChar}
\|(I_{\alpha})_{b}^{m}\chi_{E}\|_{L^{q}(\lambda^{q})}\leq c\mu^{p}(E)^{\frac{1}{p}},
\end{equation}
then $b\in \BMO_{\nu^{\frac{1}{m}}}$.
\end{thm}
In the case $m=1$ a qualitative version of this result was established by Holmes, Rahm and Spencer \cite{HRS}. Besides providing a new proof of the result in \cite{HRS}, our theorem improves that result in several directions. We provide quantitative bounds instead of qualitative ones, we extend the result to iterated commutators and we also prove that actually a restricted strong type $(p,q)$ estimate is neccesary, instead of the usual strong type $(p,q)$.

If we restrict ourselves to the case 
$\mu=\lambda$ we have the following result.
\begin{cor}\label{cor:oneweight}
Let $\alpha\in(0,n)$ and $1<p<\frac{n}{\alpha}$
, $q$ defined by $\frac{1}{q}+\frac{\alpha}{n}=\frac{1}{p}$ and $m$ a non negative integer. Assume that $w \in A_{p,q}$ and that $b\in \BMO$. Then
\begin{equation}\label{Cor1}
\|(I_\alpha)_b^mf\|_{L^{q}(w^{q})}\leq c_{n,p,q}\|b\|_{\BMO}^m[w]_{A_{p,q}}^{\left(m+1-\frac{\alpha}{n}\right)\max\left\{ 1,\frac{p'}{q}\right\} }\|f\|_{L^{p}(w^{p})},
\end{equation}
and the estimate is sharp in the sense that the inequality does not hold if we replace the exponent of the $A_{p,q}$ constant by a smaller one.

Conversely if $m>0$ and for every set $E$ of finite measure we have that
\[
\|(I_{\alpha})_{b}^{m}\chi_{E}\|_{L^{q}(w^{q})}\leq cw^{p}(E)^{\frac{1}{p}},
\]
then $b\in \BMO$.
\end{cor}

In the case $m=0$ the preceding result is due to Lacey, Moen, P\'erez and Torres \cite{LMPT}. The case $m=1$ was settled in \cite{CUM} but using a different proof based on a suitable combination of the so called conjugation method, that was introduced in \cite{CRW} (see \cite{CPP} for the first application of the method to obtain sharp constants), and an extrapolation argument. The case $m>1$ was recently established in \cite{BMMST} also relying upon the conjugation method. We observe that Corollary \ref{cor:oneweight} provides a new proof of the results in \cite{CUM,BMMST}. Additionally we settle the sharpness of the iterated case and provide a new characterization of $\BMO$ in terms of iterated commutators. The preceding result combined with the characterization recently obtained in \cite{LORR2} allows to connect the boundedness of commutators of singular integrals and of commutators of fractional integrals. For instance, if $R_j$ is any Riesz transform the following statement holds:
\[(I_\alpha)^m_b:L^q(\mathbb{R}^n)\rightarrow L^p(\mathbb{R}^n) \iff (R_j)^m_b:L^p(\mathbb{R}^n)\rightarrow L^p(\mathbb{R}^n).\]
The remainder of the paper is organized as follows. In Section \ref{Sec:Sparse} we will present and  establish the pointwise sparse domination result on which we will rely to prove Theorem \ref{thm:BloomFrac} and in Section \ref{Sec:ThmBFCor}  we provide  proofs of Theorem \ref{thm:BloomFrac} and Corollary \ref{cor:oneweight}. We end up this paper with some remarks regarding mixed estimates involving the $A_\infty$ constant.

\section{A sparse domination result for iterated commutators of fractional integrals}\label{Sec:Sparse}

We begin this section recalling the definitions of the dyadic structures we will rely upon. These definitions and a profound treatise on dyadic calculus can be found in \cite{LN}.

Given a cube $Q\subset {\mathbb R}^n$, we denote by ${\mathcal D}(Q)$ the family of all dyadic cubes with respect to $Q$, that is, the cubes obtained subdividing repeatedly $Q$ and each of its descendants into $2^n$ subcubes of the same sidelength.

Given a family of cubes  $\mathcal{D}$, we say that it is a dyadic lattice if it satisfies the following properties:
\begin{enumerate}
\item If $Q\in\mathcal{D}$, then $\mathcal{D}(Q)\subset\mathcal{D}$.
\item For every pair of cubes $Q',Q''\in \mathcal{D}$ there exists a common ancestor, namely, we can find $Q\in\mathcal{D}$ such that $Q',Q''\in {\mathcal D}(Q)$.
\item For every compact set $K\subset {\mathbb R}^n$, there exists a cube $Q\in \mathcal{D}$ such that $K\subset Q$.
\end{enumerate}

Given a dyadic lattice $\mathcal{D}$ we say that a family $\mathcal{S}\subset\mathcal{D}$ is an $\eta$-sparse family with $\eta\in(0,1)$ if there exists a family $\{E_Q\}_{Q\in\mathcal{S}}$ of pairwise disjoint measurable sets such that, for any $Q\in\mathcal{S}$, the set $E_Q$ is contained in $Q$ and satisfies $\eta |Q|\leq |E_Q|$.

 Since the first simplification of the proof of the $A_2$ theorem provided by Lerner \cite{LA2}, sparse domination theory has experienced a fruitful and fast development. However in the case of fractional integrals, the sparse domination philosophy, via dyadic discretizations of the operator, had been already implicitly exploited in \cite{SW}, \cite{P}, and a dyadic type expression for commutators had also shown up in \cite{CUM}. We remit the reader to \cite{CU} for a more detailed insight on the topic.

Relying upon ideas in \cite{IFRR} and \cite{LORR}, it is possible to obtain a pointwise sparse domination that covers the case of iterated commutators of fractional integrals. The precise statement is the following.

\begin{thm}\label{acsparse}
Let $0<\alpha<n$. Let $m$ be a non-negative integer. For every  $f\in\mathcal{C}_{c}^{\infty}(\mathbb{R}^{n})$ and $b\in L_{\text{loc }}^{m}(\mathbb{R}^{n})$,
there exist a family $\{\mathcal{D}_j\}_{j=1}^{3^{n}}$ of dyadic lattices and a family $\{\mathcal{S}_j\}_{j=1}^{3^n}$ of sparse
families such that $\mathcal{S}_{j}\subset\mathcal{D}_{j}$, for each $j$, and 
\[
|(I_{\alpha})_{b}^{m}f(x)|\leq c_{n,m,\alpha}\sum_{j=1}^{3^{n}}\sum_{h=0}^{m}\binom{m}{h}\mathcal{A}_{\alpha,\mathcal{S}_{j}}^{m,h}(b,f)(x),\qquad \text{a.e. } x\in\mathbb{R}^n,
\]
where, for a sparse family $\mathcal{S}$, $\mathcal{A}_{\alpha,\mathcal{S}}^{m,h}(b,\cdot)$ is the sparse operator given by
\[
\mathcal{A}_{\alpha,\mathcal{S}}^{m,h}(b,f)(x)=\sum_{Q\in\mathcal{S}}|b(x)-b_{Q}|^{m-h}|Q|^{\frac{\alpha}{n}}|f(b-b_{Q})^{h}|_{Q}\chi_{Q}(x).
\]
\end{thm}
To establish the preceding theorem we need to prove that the grand maximal truncated operator $\mathcal{M}_{I_{\alpha}}$ defined
by 
\[
\mathcal{M}_{I_{\alpha}}f(x)=\sup_{Q\ni x}\underset{{\scriptscriptstyle \xi\in Q}}{\esssup}\left|I_{\alpha}(f\chi_{\mathbb{R}^{n}\setminus3Q})(\xi)\right|,
\]
where the supremmum is taken over all the cubes $Q\subset\mathbb{R}^{n}$
containing $x$, maps $L^1(\mathbb{R}^n)$ to $L^{\frac{n}{n-\alpha},\infty}(\mathbb{R}^n)$. 
We will also use a local version of this operator which is defined, for a cube $Q_0\subset \mathbb{R}^n$, as
\[
\mathcal{M}_{I_{\alpha},Q_{0}}f(x)=\sup_{x\in Q\subset Q_{0}}\underset{{\scriptscriptstyle \xi\in Q}}{\esssup}\left|I_{\alpha}(f\chi_{3Q_{0}\setminus3Q})(\xi)\right|.
\]

\subsection{Lemmata}
The purpose of this subsection is to provide two lemmas that will be needed to establish Theorem \ref{acsparse}. We start presenting the first of them. 
\begin{lem}
\label{LemmaTec}Let $0<\alpha<n$. Let $Q_0\subset\mathbb{R}^n$ be a cube. The following pointwise estimates
hold:
\begin{enumerate}
\item\label{lema:1} For a.e. $x\in Q_{0}$,
\[
|I_{\alpha}(f\chi_{3Q_{0}})(x)|\leq\mathcal{M}_{I_{\alpha},Q_{0}}f(x).
\]
\item\label{lema:2} For all $x\in\mathbb{R}^{n}$
\[
\mathcal{M}_{I_{\alpha}}f(x)\leq c_{n,\alpha}\left(M_{\alpha}f(x)+I_{\alpha}|f|(x)\right).
\]
From this last estimate it follows that $\mathcal{M}_{I_{\alpha}}$ is bounded from $L^1(\mathbb{R}^n)$ to $L^{\frac{n}{n-\alpha},\infty}(\mathbb{R}^n)$. 
\end{enumerate}
\end{lem}

\begin{proof}[Proof of Lemma \ref{LemmaTec}]
To prove \eqref{lema:1}, let $Q(x,s)$ be a cube centered at $x$
and such that $Q(x,s)\subset Q_{0}$. Then,
\begin{equation}\label{lema2}
\begin{split}|I_{\alpha}(f\chi_{3Q_{0}})(x)| & \leq|I_{\alpha}(f\chi_{3Q(x,s)})(x)|+|I_{\alpha}(f\chi_{3Q_{0}\setminus3Q(x,s)})(x)|\\
 & \leq|I_{\alpha}(f\chi_{3Q(x,s)})(x)|+\mathcal{M}_{I_{\alpha},Q_{0}}f(x)\\
 & \leq C_{n,\alpha}s^{\alpha}Mf(x)+\mathcal{M}_{I_{\alpha},Q_{0}}f(x),
\end{split}
\end{equation}
where the last estimate for the first term follows by standard computations involving a dyadic annuli-type decomposition of the cube $Q(x,s)$. 
The estimate in \eqref{lema:1} is then settled letting $s\rightarrow0$ in \eqref{lema2}.

For the proof of the pointwise inequality in \eqref{lema:2}, let $x$ be a point in $\mathbb{R}^n$ and $Q$ a cube containing $x$. Denote by $B_{x}$
the closed ball centered at $x$ of radius $2\operatorname{diam} Q$. Then
$3Q\subset B_{x}$, and, for every $\xi\in Q$ we obtain 
\[
\begin{split}|I_{\alpha}(f\chi_{\mathbb{R}^{n}\setminus3Q})(\xi)| & =|I_{\alpha}(f\chi_{\mathbb{R}^{n}\setminus B_{x}})(\xi)+I_{\alpha}(f\chi_{B_{x}\setminus3Q})(\xi)|\\
 & \leq |I_{\alpha}(f\chi_{\mathbb{R}^{n}\setminus B_{x}})(\xi)-I_{\alpha}(f\chi_{\mathbb{R}^{n}\setminus B_{x}})(x)|\\
 & +|I_{\alpha}(f\chi_{B_{x}\setminus3Q})(\xi)|+|I_{\alpha}(f\chi_{\mathbb{R}^{n}\setminus B_{x}})(x)|.
\end{split}
\]
For the first term, by using the mean value theorem and adapting \cite[Theorem 2.1.10]{GCFA} to our setting, we get
\[
\begin{split}  |I_{\alpha}(f\chi_{\mathbb{R}^{n}\setminus B_{x}})(\xi)-I_{\alpha}(f\chi_{\mathbb{R}^{n}\setminus B_{x}})(x)|
 & \leq\int_{\mathbb{R}^{n}\setminus B_{x}}\left|\frac{1}{|y-\xi|^{n-\alpha}}-\frac{1}{|y-x|^{n-\alpha}}\right||f(y)|dy\\
 & \leq c_{n,\alpha}\int_{\mathbb{R}^{n}\setminus B_{x}}\frac{|x-\xi|}{(|x-y|+|y-\xi|)^{n-\alpha+1}}|f(y)|dy\\
 & \leq c_{n,\alpha}M_{\alpha}f(x).
\end{split}
\]
For the second term, taking into account the definition of $B_x$, we can write
\[
\begin{split}|I_{\alpha}(f\chi_{B_{x}\setminus3Q})(\xi)| & =\left|\int_{B_{x}\setminus3Q}\frac{1}{|y-\xi|^{n-\alpha}}f(y)dy\right|\\
 & \leq\int_{B_{x}\setminus3Q}\frac{1}{|y-\xi|^{n-\alpha}}\left|f(y)\right|dy\\
 & \leq c_{n,\alpha}\frac{1}{\ell(Q)^{n-\alpha}}\int_{B_{x}}|f(y)|dy\\
 & \leq c_{n,\alpha}M_{\alpha}f(x).
\end{split}
\]
To end the proof of this pointwise estimate we observe that 
\[
|I_{\alpha}(f\chi_{\mathbb{R}^{n}\setminus B_{x}})(x)|\leq I_{\alpha}|f|(x),
\]
which finishes the proof of \eqref{lema:2}.
Now, taking into account the pointwise estimate we have just obtained, and the boundedness properties of the operators $I_{\alpha}$ and $M_\alpha$, it is clear that $\mathcal{M}_{I_{\alpha}}$ is bounded from $L^1(\mathbb{R}^n)$ to $L^{\frac{n}{n-\alpha},\infty}(\mathbb{R}^n)$, and we are done.
\end{proof}

The second lemma that we will need for the proof of Theorem \ref{acsparse} is the so called $3^n$ dyadic lattices trick. This result was established in \cite{LN} and essentially says that given a dyadic lattice $\mathcal{D}$, if we consider the family of cubes $\{3Q\,:\,Q\in\mathcal{D}\}$ it is possible to arrange them in $3^n$ dyadic lattices. 
\begin{lem}
Given a dyadic lattice $\mathcal{D}$ there exist $3^{n}$ dyadic
lattices $\mathcal{D}_{j}$ such that 
\[
\{3Q\,:\,Q\in\mathcal{D}\}=\bigcup_{j=1}^{3^{n}}\mathcal{D}_{j}
\]
and for every cube $Q\in\mathcal{D}$ we can find a cube $R_{Q}$
in each $\mathcal{D}_{j}$ such that $Q\subset R_{Q}$ and $3\ell({Q})=\ell({R_{Q}})$
\end{lem}
\begin{rem}
\label{Rem}Fix a dyadic lattice $\mathcal{D}$. For an arbitrary cube $Q\subset\mathbb{R}^{n}$
there is a cube $Q'\in\mathcal{D}$ such that $\frac{\ell({Q})}{2}<\ell(Q')\leq \ell({Q})$
and $Q\subset3Q'$. We can take a cube with that property
since every generation of cubes in $\mathcal{D}$ tiles $\mathbb{R}^n$.
From this and the preceding lemma it follows that $3Q'=P\in\mathcal{D}_{j}$
for some $j\in\{1,\dots,3^{n}\}$. Therefore, for every cube $Q\subset\mathbb{R}^{n}$
there exists  some $j\in\{1,\dots,3^{n}\}$ and some $P\in\mathcal{D}_{j}$ such that $Q\subset P$ and
$\ell({P})\leq3\ell({Q})$ and consequently $|Q|\leq|P|\leq3^{n}|Q|$. 

\end{rem}

\subsection{Proof of Theorem \ref{acsparse}}

From Remark \ref{Rem} it follows that there exist $3^{n}$ dyadic
lattices such that for every cube $Q$ of $\mathbb{R}^{n}$ there
is a cube $R_{Q}\in\mathcal{D}_{j}$ for some $j$ for which $3Q\subset R_{Q}$
and $|R_{Q}|\leq9^{n}|Q|$.

We claim that there is a positive constant $c_{n,m,\alpha}$ verifying that, for any cube $Q_{0}\subset\mathbb{R}^{n}$,  there
exists a $\frac{1}{2}$-sparse family $\mathcal{F}\subset\mathcal{D}(Q_{0})$
such that for a.e. $x\in Q_{0}$ 
\begin{equation}
\left|(I_{\alpha})_{b}^{m}(f\chi_{3Q_{0}})(x)\right|\leq c_{n,m,\alpha}\sum_{h=0}^{m}\binom{m}{h}\mathcal{B}_{\mathcal{F}}^{m,h}(b,f)(x),\label{eq:Claim-1-1}
\end{equation}
where 
\[
\mathcal{B}_{\mathcal{F}}^{m,h}(b,f)(x)=\sum_{Q\in\mathcal{F}}|b(x)-b_{R_{Q}}|^{m-h}|3Q|^{\frac{\alpha}{n}}|f(b-b_{R_{Q}})^{h}|_{3Q}\chi_{Q}(x).
\]

Suppose that we have already proved the claim. Let us take a partition
of $\mathbb{R}^{n}$ by a family $\{Q_k\}_{k\in\mathbb{N}}$ of cubes $Q_{k}$ such that $\supp(f)\subset3Q_{k}$
for each $j\in\mathbb{N}$. We can do it as follows. We start with a cube $Q_{0}$
such that $\supp(f)\subset Q_{0}.$ And cover $3Q_{0}\setminus Q_{0}$
by $3^{n}-1$ congruent cubes $Q_{k}$. Each of them satisfies $Q_{0}\subset3Q_{k}$.
We do the same for $9Q_{0}\setminus3Q_{0}$ and so on. The union of
all those cubes, including $Q_{0}$, will satisfy the desired properties.

Fix $k\in \mathbb{N}$ and apply the claim to the cube $Q_{k}$. Then we have that since
$\supp f\subset3Q_{k}$ the following estimate holds for almost every $x\in\mathbb{R}^n$:
\[
\left|(I_{\alpha})_{b}^{m}f(x)\right|\chi_{Q_{k}}(x)=\left|(I_{\alpha})_{b}^{m}(f\chi_{3Q_{k}})(x)\right|\chi_{Q_{k}}(x)\leq c_{n,m,\alpha}\sum_{h=0}^{m}\binom{m}{h}\mathcal{B}_{\mathcal{F}_{k}}^{m,h}(b,f)(x),
\]
where $\mathcal{F}_{k}\subset\mathcal{D}(Q_{k})$ is a $\frac{1}{2}$-sparse
family. Taking $\mathcal{F}=\bigcup_{k\in\mathbb{N}}\mathcal{F}_{k}$ we have that
$\mathcal{F}$ is a $\frac{1}{2}$-sparse family and
\[
\left|(I_{\alpha})_{b}^{m}f(x)\right|\leq c_{n,m,\alpha}\sum_{h=0}^{m}\binom{m}{h}\mathcal{B}_{\mathcal{F}}^{m,h}(b,f)(x),\qquad \text{a.e. } x\in\mathbb{R}^n.
\]
Fix $Q\subset \mathcal{F}$. Since $3Q\subset R_{Q}$ and $|R_{Q}|\leq3^{n}|3Q|$, we have that
$|3Q|^{\frac{\alpha}{n}}|f(b-b_{R_{Q}})^{h}|_{3Q}\leq 3^n|R_{Q}|^{\frac{\alpha}{n}}|f(b-b_{R_{Q}})^{h}|_{R_Q}$.
Setting 
\[
\mathcal{S}_{j}=\left\{ R_{Q}\in\mathcal{D}_{j}\,:\,Q\in\mathcal{F}\right\} 
\]
and using that $\mathcal{F}$ is $\frac{1}{2}$-sparse, we obtain
that each family $\mathcal{S}_{j}$ is $\frac{1}{2\cdot9^{n}}$-sparse.
Then we have that 
\[
\left|(I_{\alpha})_{b}^{m}f(x)\right|\leq c_{n,m,\alpha}\sum_{j=1}^{3^{n}}\sum_{h=0}^{m}\binom{m}{h}\mathcal{A}_{\alpha,\mathcal{S}_{j}}^{m,h}(b,f)(x),\qquad \text{a.e. }x\in\mathbb{R}^n,
\]
and we are done.

\subsubsection*{Proof of the Claim (\ref{eq:Claim-1-1})}

To prove the claim it suffices to prove the following recursive estimate:
there is a positive constant $c_{n,m,\alpha}$ verifying that there exists a countable family $\{P_j\}_{j}$ of  pairwise disjoint cubes in $\mathcal{D}(Q_{0})$
such that $\sum_{j}|P_{j}|\leq\frac{1}{2}|Q_{0}|$ and 
\begin{equation*}
\begin{aligned}
 |(I_{\alpha})_{b}^{m}&(f\chi_{3Q_{0}})(x)|\chi_{Q_{0}}(x)\\
 & \le  c_{n,m,\alpha}\sum_{h=0}^{m}\binom{m}{h}|b(x)-b_{R_{Q_{0}}}|^{m-h}|3Q_{0}|^{\frac{\alpha}{n}}|f(b-b_{R_{Q_{0}}})^{h}|_{3Q_{0}}\chi_{Q_{0}}(x)\\
 & +  \sum_{j}|(I_{\alpha})_{b}^{m}(f\chi_{3P_{j}})(x)|\chi_{P_{j}}(x),\qquad \text{a.e. }x\in Q_0.
 \end{aligned}
\end{equation*}
 Iterating this estimate, we obtain \eqref{eq:Claim-1-1}
with $\mathcal{F}=\{P_{j}^{k}\}_{j,k}$ where $\{P_{j}^{0}\}_{j}:=\{Q_{0}\}$,
$\{P_{j}^{1}\}_{j}:=\{P_{j}\}_{j}$ and $\{P_{j}^{k}\}_{j}$ is the union of the cubes obtained
at the $k$-th stage of the iterative process from each of the cubes $P_j^{k-1}$ of the $(k-1)$-th stage. Clearly
$\mathcal{F}$ is a $\frac{1}{2}$-sparse family, since the conditions in the definition hold for the sets
\[
E_{P_{j}^{k}}=P_{j}^{k}\setminus\bigcup_{j}P_{j}^{k+1}.
\]
Let us prove then the recursive estimate.

For any countable family $\{P_j\}_j$ of disjoint cubes $P_{j}\in\mathcal{D}(Q_{0})$
we have that 

\[
\begin{split} & \left|(I_{\alpha})_{b}^{m}(f\chi_{3Q_{0}})(x)\right|\chi_{Q_{0}}(x)\\
= & \left|(I_{\alpha})_{b}^{m}(f\chi_{3Q_{0}})(x)\right|\chi_{Q_{0}\setminus\bigcup_{j}P_{j}}(x)+\sum_{j}\left|(I_{\alpha})_{b}^{m}(f\chi_{3Q_{0}})(x)\right|\chi_{P_{j}}(x)\\
\leq & \left|(I_{\alpha})_{b}^{m}(f\chi_{3Q_{0}})(x)\right|\chi_{Q_{0}\setminus\bigcup_{j}P_{j}}(x)+\sum_{j}\left|(I_{\alpha})_{b}^{m}(f\chi_{3Q_{0}\setminus3P_{j}})(x)\right|\chi_{P_{j}}(x)\\
 & +\sum_{j}\left|(I_{\alpha})_{b}^{m}(f\chi_{3P_{j}})(x)\right|\chi_{P_{j}}(x)
\end{split}
\]
for almost every $x\in\mathbb{R}^n$.
So it suffices to show that we can find a positive constant $c_{n,m,\alpha}$ in such a way we can choose a countable family $\{P_j\}_j$ of pairwise disjoint
cubes in $\mathcal{D}(Q_{0})$ with $\sum_{j}|P_{j}|\leq\frac{1}{2}|Q_{0}|$
and such that, for a.e. $x\in Q_{0}$,
\begin{equation}
\begin{split} & \left|(I_{\alpha})_{b}^{m}(f\chi_{3Q_{0}})(x)\right|\chi_{Q_{0}\setminus\bigcup_{j}P_{j}}(x)+\sum_{j}\left|(I_{\alpha})_{b}^{m}(f\chi_{3Q_{0}\setminus3P_{j}})(x)\right|\chi_{P_{j}}(x)\\
 & \leq c_{n,m,\alpha}\sum_{h=0}^{m}\binom{m}{h}|b(x)-b_{R_{Q_{0}}}|^{m-h}|3Q_{0}|^{\frac{\alpha}{n}}|f(b-b_{R_{Q_{0}}})^{h}|_{3Q_{0}}\chi_{Q_{0}}(x)
\end{split}
\label{eq:3.2}
\end{equation}
Using that $(I_{\alpha})_{b}^{m}f=(I_{\alpha})_{b-c}^{m}f$ for any
$c\in\mathbb{R}$, and also that 
\[
(I_{\alpha})_{b-c}^{m}f=\sum_{h=0}^{m}(-1)^{h}\binom{m}{h}I_{\alpha}((b-c)^{h}f)(b-c)^{m-h},
\]
it follows that
\begin{eqnarray}
 &  & |(I_{\alpha})_{b}^{m}(f\chi_{3Q_{0}})|\chi_{Q_{0}\setminus\cup_{j}P_{j}}(x)+\sum_{j}|(I_{\alpha})_{b}^{m}(f\chi_{3Q_{0}\setminus3P_{j}})|\chi_{P_{j}}(x)\nonumber \\
 &  & \le\sum_{h=0}^{m}\binom{m}{h}|b(x)-b_{R_{Q_{0}}}|^{m-h}|I_{\alpha}((b-b_{R_{Q_{0}}})^{h}f\chi_{3Q_{0}})(x)|\chi_{Q_{0}\setminus\cup_{j}P_{j}}(x)\label{eq:NoPj-1}\\
 &  & +\sum_{h=0}^{m}\binom{m}{h}|b(x)-b_{R_{Q_{0}}}|^{m-h}\sum_{j}|I_{\alpha}((b-b_{R_{Q_{0}}})^{h}f\chi_{3Q_{0}\setminus3P_{j}})(x)|\chi_{P_{j}}(x).\label{eq:SumPj-1}
\end{eqnarray}
Now we define the set $E=\cup_{h=0}^{m}E_{h}$, where 
\[
E_{h}=\left\{ x\in Q_{0}\,:\,\mathcal{M}_{I_{\alpha},Q_{0}}\left((b-b_{R_{Q_{0}}})^{h}f\right)(x)>c_{n,m,\alpha}|3Q_{0}|^{\frac{\alpha}{n}}|(b-b_{R_{Q_{0}}})^{h}f|{}_{3Q_{0}}\right\}, 
\]
with $c_{n,m,\alpha}$ being a positive number to be chosen.

As we proved in Lemma \ref{LemmaTec}, we have that
\[
c_{n,\alpha}:=\|\mathcal{M}_{I_{\alpha}}\|_{L^{1}(\mathbb{R}^n)\rightarrow L^{\frac{n}{n-\alpha},\infty}(\mathbb{R}^n)}<\infty,
\]
so, since $\mathcal{M}_{I_{\alpha},Q_{0}}g\leq \mathcal{M}_{I_{\alpha}}(g\chi_{3Q_0})$, we can write, for each $h\in\{0,1,\ldots,m\}$, 
\[
\begin{split}|E_{h}| & \leq\left(\frac{c_{n,\alpha}\int_{3Q_{0}}|f(b-b_{R_{Q_{0}}})^{h}|}{c_{n,m,\alpha}|3Q_{0}|^{\frac{\alpha}{n}}|f(b-b_{R_{Q_{0}}})^{h}|{}_{3Q_{0}}}\right)^{\frac{n}{n-\alpha}}\\
 & =\left(\frac{c_{n,\alpha}|3Q_{0}|^{\frac{\alpha}{n}-1}\int_{3Q_{0}}|f(b-b_{R_{Q_{0}}})^{h}|}{c_{n,m,\alpha}|3Q_{0}|^{\frac{\alpha}{n}}|f(b-b_{R_{Q_{0}}})^{h}|{}_{3Q_{0}}}\right)^{\frac{n}{n-\alpha}}|3Q_{0}|^{\left(1-\frac{\alpha}{n}\right)\frac{n}{n-\alpha}}\\
 & =\left(\frac{c_{n,\alpha}}{c_{n,m,\alpha}}\right)^{\frac{n}{n-\alpha}}|3Q_{0}|^{\left(1-\frac{\alpha}{n}\right)\frac{n}{n-\alpha}}=3^{n}\left(\frac{c_{n,\alpha}}{c_{n,m,\alpha}}\right)^{\frac{n}{n-\alpha}}|Q_{0}|,
\end{split}
\]
and we can choose $c_{n,m,\alpha}$ such that 
\begin{equation}
|E|\leq \sum_{h=0}^m|E_h|\leq\frac{1}{2^{n+2}}|Q_{0}|,\label{eq:EQ0}
\end{equation}
this choice being independent from $Q_0$ and $f$.

Now we apply Calder\'on-Zygmund decomposition to the function $\chi_{E}$
on $Q_{0}$ at height $\lambda=\frac{1}{2^{n+1}}$. We obtain a countable family $\{P_j\}_{j}$ of  pairwise
disjoint cubes in $\mathcal{D}(Q_{0})$ such that 
\[
\chi_{E}(x)\leq\frac{1}{2^{n+1}},\qquad \text{a.e. }x\not\in\bigcup_j P_{j}.
\]
From this it follows that $\left|E\setminus\bigcup_{j}P_{j}\right|=0$.
The family $\{P_j\}_j$ also satisfies that 
\[
\sum_{j}|P_{j}|\leq2^{n+1}|E|\leq\frac{1}{2}|Q_{0}|
\]
and
\[
\frac{|P_{j}\cap E|}{|P_{j}|}=\frac{1}{|P_{j}|}\int_{P_{j}}\chi_{E}(x)\leq\frac{1}{2}\qquad\text{ for all }j,
\]
from which it readily follows that $|P_{j}\cap E^{c}|>0$ for every $j$. Indeed, given $j$,
\[
|P_{j}|=|P_{j}\cap E|+|P_{j}\cap E^{c}|\leq\frac{1}{2}|P_{j}|+|P_{j}\cap E^{c}|,
\]
and from this it follows that $0<\frac{1}{2}|P_{j}|<|P_{j}\cap E^{c}|$. 

Fix some $j$. Since we have $P_{j}\cap E^{c}\neq\emptyset$, we observe that 
\[
\mathcal{M}_{I_{\alpha},Q_{0}}\left((b-b_{R_{Q_{0}}})^{h}f\right)(x)\leq c_{n,m,\alpha}|3Q_{0}|^{\frac{\alpha}{n}}|(b-b_{R_{Q_{0}}})^{h}f|{}_{3Q_{0}}
\]
for some $x\in P_{j}$ and this implies that, for any $Q\subset Q_0$ containing $x$, we have
\[
\underset{{\scriptscriptstyle \xi\in Q}}{\esssup}\left|I_{\alpha}((b-b_{R_{Q_{0}}})^{h}f\chi_{3Q_{0}\setminus3Q})(\xi)\right|\leq c_{n,m,\alpha}|3Q_{0}|^{\frac{\alpha}{n}}|(b-b_{R_{Q_{0}}})^{h}f|{}_{3Q_{0}},
\]
which allows us to control the summation in (\ref{eq:SumPj-1}) by considering the cube $P_j$.

Now, by \eqref{lema:1} in Lemma \ref{LemmaTec}, we know that 
\[
\left|I_{\alpha}((b-b_{R_{Q_{0}}})^{h}f\chi_{3Q_{0}})(x)\right|\leq\mathcal{M}_{I_{\alpha},Q_{0}}\left((b-b_{R_{Q_{0}}})^{h}f\right)(x),\qquad \text{a.e. }x\in Q_{0}.
\]
Since  $\left|E\setminus\bigcup_{j}P_{j}\right|=0$ we have, by the definition of $E$, that
\[
\mathcal{M}_{I_{\alpha},Q_{0}}\left((b-b_{R_{Q_{0}}})^{h}f\right)(x)\leq c_{n,m,\alpha}|3Q_{0}|^{\frac{\alpha}{n}}|(b-b_{R_{Q_{0}}})^{h}f|{}_{3Q_{0}},\quad \text{a.e. }x\in Q_0\setminus\bigcup_{j}P_{j}.
\]
Consequently,
\[
\left|I_{\alpha}((b-b_{R_{Q_{0}}})^{h}f\chi_{3Q_{0}})(x)\right|\leq c_{n,m,\alpha}|3Q_{0}|^{\frac{\alpha}{n}}|(b-b_{R_{Q_{0}}})^{h}f|{}_{3Q_{0}},\quad \text{a.e. }x\in Q_0\setminus\bigcup_{j}P_{j}.
\]
These estimates allow us to control the remaining terms in (\ref{eq:NoPj-1}),
so we are done.

\section{Proofs of Theorem \ref{thm:BloomFrac} and Corollary \ref{cor:oneweight}}\label{Sec:ThmBFCor}
The proof of Theorem  \ref{thm:BloomFrac} is presented in the two first subsections. First we deal with the upper bound and then with the necessity.

We will end up this section with a subsection devoted to establish Corollary \ref{cor:oneweight}.

\subsection{Proof of the upper bound}
To settle the upper bound in Theorem \ref{thm:BloomFrac} we argue as in \cite[Theorem 1.4]{LORR} or, to be more precise as in \cite[Theorem 1.1]{LORR2}. To do that we need to borrow the following estimate that was obtained in the case $j=1$ in \cite{LORR} and for $j>1$ in \cite{LORR2} and that can be stated as follows.
\begin{lem}\label{Lem:OutBMO}
Let $\mathcal{S}$ be a sparse family contained in a dyadic lattice $\mathcal{D}$, $\eta$ a weight, $b\in \BMO_\eta$ and $f\in\mathcal{C}_c^\infty(\mathbb{R}^n)$. There exists a possibly larger sparse family $\tilde{\mathcal{S}}\subset\mathcal{D}$ containing $\mathcal{S}$ such that, for every positive integer $j$ and every $Q\in\tilde{\mathcal{S}}$
\[\int_Q|b-b_Q|^j|f|\leq c_n \|b\|_{\BMO_\eta}^j \int_Q A^j_{\tilde{\mathcal{S}},\eta}f\]
where  $A^j_{\tilde{\mathcal{S}},\eta}f$ stands for the $j$-th iteration of $A_{\tilde{\mathcal{S}},\eta}$, which is defined by $A_{\tilde{\mathcal{S}},\eta}f:=A_{\tilde{\mathcal{S}}}(f)\eta$, with $A_{\tilde{\mathcal{S}}}$ being the sparse operator given by
\[A_{\tilde{\mathcal{S}}}f(x)=\sum_{Q\in\tilde{\mathcal{S}}}\frac{1}{|Q|}\int_Q |f|\chi_Q(x).\]
\end{lem}


We will also make use of the following quantitative estimates.
Let $1<p<\infty$ and $\mathcal{S}$ a $\gamma$-sparse family. If $w\in A_p$ then
\begin{equation}\label{eq:Sparse}
\|A_\mathcal{S}\|_{L^p(w)}\leq c_{n,p}[w]_{A_p}^{\max\left\{1,\frac{1}{p-1}\right\}}.\end{equation}
If $p,q,\alpha$ are as in the hypothesis of Theorem \ref{thm:BloomFrac} and $w\in A_{p,q}$, then
\begin{equation}\label{eq:Frac}
\|I_\mathcal{S}^\alpha\|_{L^q(w^q)\rightarrow L^p(w^p)}\leq c_{n,p,q,\alpha}[w]_{A_{p,q}}^{\left(1-\frac{\alpha}{n}\right)\max\left\{1,\frac{p}{q'}\right\}},
\end{equation}
where \[I_\mathcal{S}^\alpha f(x)=\sum_{Q\in\mathcal{S}}\frac{1}{|Q|^{1-\frac{\alpha}{n}}}\int_Q |f|\chi_Q(x).\]
We observe that the proof of \eqref{eq:Frac} is implicit in one of the proofs of \cite[Theorem 2.6]{LMPT} that relies essentially on computing the norm of the operator $I_\mathcal{S}^\alpha$ by duality.

At this point we are in the position to prove the estimate  \eqref{eq:depIalphabm}.

We assume by now that $b \in L^m_{\mathrm{loc}}(\mathbb{R}^n)$ and we prove in the end that this assumption is always true. Taking into account Theorem \ref{acsparse}, it suffices to prove the estimate for the sparse operators 
\[A_{\alpha,\mathcal{S}}^{m,h}(b,f):=\sum_{Q\in\mathcal{S}}|b-b_Q|^{m-h}|Q|^{\alpha/n}|f(b-b_Q)^h|_Q\chi_Q,\qquad h\in\{0,1,\ldots,m\}.\]
Assume that $b\in \BMO_\eta$ with $\eta$ to be chosen. 
We observe that, using Lemma \ref{Lem:OutBMO},
\[
\begin{split}
\left|\int_{\mathbb{R}^n} A_{\alpha,\mathcal{S}}^{m,h}(b,f)g\lambda^q\right|&\leq\sum_{Q\in\mathcal{S}}\left(\int_Q|g||b-b_Q|^{m-h}\lambda^q\right)\frac{1}{|Q|^{1-\alpha/n}}\int_Q|b-b_Q|^h|f|.\\
& \leq c_n\|b\|_{\BMO_\eta}^m\sum_{Q\in\mathcal{S}}\left(\int_QA_{\tilde{\mathcal{S}},\eta}^{m-h}(|g|\lambda^q)\right)\frac{1}{|Q|^{1-\alpha/n}}\int_QA_{\tilde{\mathcal{S}},\eta}^{h}(|f|)\\
 &\leq c_n\|b\|_{\BMO_\eta}^m \int_{\mathbb{R}^n}\sum_{Q\in\mathcal{S}}\frac{1}{|Q|^{1-\alpha/n}}\left(\int_Q A_{\tilde{\mathcal{S}},\eta}^{h}(|f|)\right)\chi_QA_{\tilde{\mathcal{S}},\eta}^{m-h}(|g|\lambda^q)\\
 &=c_n\|b\|_{\BMO_\eta}^m\int_{\mathbb{R}^n} I_{\mathcal{S}}^\alpha\left[A_{\tilde{\mathcal{S}},\eta}^h(|f|)\right](x)A_{\tilde{\mathcal{S}},\eta}^{m-h}(|g|\lambda^q)(x)dx.
\end{split}
\]
Let us call $I_{\mathcal{S},\eta}^\alpha f:=I_{\mathcal{S}}^\alpha (f) \eta$. Now, the self-adjointness of $A_{\tilde{\mathcal{S}}}$ yields
\[
\begin{split}
\int_{\mathbb{R}^n}I_{\mathcal{S}}^\alpha\left(A_{\tilde{\mathcal{S}},\eta}^h(|f|)\right)A_{\tilde{\mathcal{S}},\eta}^{m-h}(|g|\lambda^q)
=\int_{\mathbb{R}^n} A_{\tilde{\mathcal{S}}}\left\{A_{\tilde{\mathcal{S}},\eta}^{m-h-1}\left[I_{\mathcal{S},\eta}^\alpha\left(A_{\tilde{\mathcal{S}},\eta}^h(|f|)\right)\right]\right\}|g|\lambda^q.
\end{split}
\]
Combining the preceding estimates we have that
\[\left|\int_{\mathbb{R}^n} A_{\alpha,\mathcal{S}}^{m,h}(b,f)g\lambda^q\right|\leq c_n \|b\|_{\BMO_\eta}^m
\left\|A_{\tilde{\mathcal{S}}}A_{\tilde{\mathcal{S}},\eta}^{m-h-1}I_{\mathcal{S},\eta}^\alpha A_{\tilde{\mathcal{S}},\eta}^h(|f|)\right\|_{L^q(\lambda^q)}\|g\|_{L^{q'}(\lambda^q)}\]
and consequently, taking supremum on $g\in L^{q'}(\lambda^q)$ with $\|g\|_{L^{q'}(\lambda^q)}=1$,
\[
\begin{split}
\|A_{\alpha,\mathcal{S}}^{m,h}(b,f)\|_{L^q(\lambda^q)}\leq c_n\|b\|_{\BMO_\eta}^m \left\|A_{\tilde{\mathcal{S}}}A_{\tilde{\mathcal{S}},\eta}^{m-h-1}I_{\mathcal{S},\eta}^\alpha A_{\tilde{\mathcal{S}},\eta}^h(|f|)\right\|_{L^q(\lambda^q)}.
\end{split}
\]
Taking into account \eqref{eq:Sparse}
\[
\begin{split}
 \left\|A_{\tilde{\mathcal{S}}}A_{\tilde{\mathcal{S}},\eta}^{m-h-1}\right.&\left.I_{\mathcal{S},\eta}^\alpha A_{\tilde{\mathcal{S}},\eta}^h(|f|)\right\|_{L^q(\lambda^q)}\\
&\leq c_{n,q}\left(\prod_{l=0}^{m-h-1}[\lambda^{q}\eta^{lq}]_{A_q}\right)^{\max\left\{1,\frac{1}{q-1}\right\}}\left\|I_{\mathcal{S},\eta}^\alpha A_{\tilde{\mathcal{S}},\eta}^h(|f|)\right\|_{L^q(\lambda^q\eta^{q(m-h-1)})}.
\end{split}
\]
Using \eqref{eq:Frac}, we have that
\[
\begin{split}
&\left\|I_{\mathcal{S},\eta}^\alpha A_{\tilde{\mathcal{S}},\eta}^h(|f|)\right\|_{L^q(\lambda^q\eta^{q(m-h-1)})}
=\left\|I_{\mathcal{S}}^\alpha A_{\tilde{\mathcal{S}},\eta}^h(|f|)\right\|_{L^q(\lambda^q\eta^{q(m-h)})}\\
&\leq c_{n,p,\alpha}[\lambda\eta^{m-h}]^{\left(1-\frac{\alpha}{n}\right)\max\left\{1,\frac{p'}{q}\right\}}_{A_{p,q}}\left\|A_{\tilde{\mathcal{S}},\eta}^h(|f|)\right\|_{L^p(\lambda^p\eta^{p(m-h)})}
\end{split}
\]
and applying again \eqref{eq:Sparse},
\[\left\|A_{\tilde{\mathcal{S}},\eta}^h(|f|)\right\|_{L^p(\lambda^p\eta^{p(m-h)})}\leq c_{n,p} \left(\prod_{l=m-h+1}^m[\lambda^p\eta^{lp}]_{A_p}\right)^{\max\left\{1,\frac{1}{p-1}\right\}}\|f\|_{L^p(\lambda^p\eta^{mp})}.\]
Gathering the preceding estimates we have that
\[\|A_{\alpha,\mathcal{S}}^{m,h}(b,f)\|_{L^q(\lambda^q)} \leq c_{n,\alpha,p}\|b\|_{\BMO_{\eta}}^m PQ [\lambda\eta^{m-h}]^{\left(1-\frac{\alpha}{n}\right)\max\left\{1,\frac{p'}{q}\right\}}_{A_{p,q}}\|f\|_{L^p(\lambda^p\eta^{mp})},\]
where
\[P=\left(\prod_{l=0}^{m-h-1}[\lambda^{q}\eta^{lq}]_{A_q}\right)^{\max\left\{1,\frac{1}{q-1}\right\}} \qquad Q=\left(\prod_{l=m-h+1}^m[\lambda^p\eta^{lp}]_{A_p}\right)^{\max\left\{1,\frac{1}{p-1}\right\}}.\]
Now we observe that choosing $\eta=\nu^{1/m}$, it readily follows from H\"older's inequality
\[
\begin{split}
[\lambda\nu^{\frac{m-h}{m}}]_{A_{p,q}}\leq [\lambda]_{A_{p,q}}^{\frac{h}{m}}[\mu]_{A_{p,q}}^{\frac{m-h}{m}}\quad\text{and}\quad [\lambda^r\nu^{r\frac{l}{m}}]_{A_r}\leq [\lambda^r]_{A_r}^{\frac{m-l}{m}}[\mu^r]_{A_r}^{\frac{l}{m}},\quad r=p,q.
\end{split}
\] 
Thus, we can write 
\[P\leq \left(\prod_{l=0}^{m-h-1}[\lambda^q]_{A_q}^{\frac{m-l}{m}}[\mu^q]_{A_q}^{\frac{l}{m}}\right)^{\max\{1,\frac{1}{q-1}\}} \quad Q\leq \left(\prod_{l=m-h+1}^{m}[\lambda^p]_{A_p}^{\frac{m-l}{m}}[\mu^p]_{A_p}^{\frac{l}{m}}\right)^{\max\{1,\frac{1}{p-1}\}}\]
and, computing the products, we obtain
\[
P\leq \left([\lambda^q]_{A_q}^{\frac{m+(h+1)}{2}}[\mu^q]_{A_q}^{\frac{m-(h+1)}{2}}\right)^{\frac{m-h}{m}\max\{1,\frac{1}{q-1}\}}
\]
and\[
Q\leq \left([\lambda^p]_{A_p}^{\frac{h-1}{2}}[\mu^p]_{A_p}^{m-\frac{h-1}{2}}\right)^{\frac{h}{m}\max\{1,\frac{1}{p-1}\}}.
\]

Combining all the preceding estimates leads to the desired estimate.

To end the proof we are going to show that $b\in L^m_{\text{loc}}(\mathbb{R}^n)$.
Indeed, for any compact set $K$ we choose a cube $Q$ such that $K\subset Q$. Then
\[\int_K|b|^m\leq \int_Q|b|^m\leq c_m\int_Q|b-b_Q|^m+c_m\left(\int_Q|b|\right)^m\]
Since $b$ is locally integrable, we only have to the deal with the first term.
We observe that by Lemma \ref{Lem:OutBMO},

\[
\begin{split}
\int_Q|b-b_Q|^m\chi_Q&\leq c_n \|b\|_{\BMO_{\nu^{1/m}}}^m \int_Q A^m_{\tilde{\mathcal{S}},\nu^{1/m}}(\chi_Q)
\\
&\leq c_n \|b\|_{\BMO_{\nu^{1/m}}}^m\left( \int_{\mathbb{R}^n} A^m_{\tilde{\mathcal{S}},\nu^{1/m}}(\chi_Q)^p\lambda^p\right)^{\frac{1}{p}}\left(\int_Q \lambda^{\frac{p}{1-p}}\right)^\frac{1}{p'}
\end{split}
\]
and arguing analogously as above we are done.

\subsection{Proof of the necessity}

We are going to follow ideas in \cite{LORR}. First we recall \cite[Lemma 2.1]{LORR}
\begin{lem}\label{BMO}
Let $\eta\in A_{\infty}$. Then
\[
\|b\|_{\BMO_{\eta}}\leq\sup_{Q}\omega_{\lambda}(b,Q)\frac{|Q|}{\eta(Q)}\qquad0<\lambda<\frac{1}{2^{n+1}}.
\]
where $\omega_{\lambda}(f,Q)=\inf_{c\in\mathbb{R}}\left((f-c)\chi_{Q}\right)^{*}(\lambda|Q|)$
and 
\[((f-c)\chi_Q)^*(\lambda|Q|)=\sup_{\stackrel{\scriptstyle E\subset Q}{|E|=\lambda|Q|}}\inf_{x\in E} |(f-c)(x)|.\]
\end{lem}

Armed with that lemma we are in the position to provide a proof of the necessity. Let $Q\subset\mathbb{R}^{n}$ be an arbitrary cube. There exists
a subset $E\subset Q$ with $|E|=\frac{1}{2^{n+2}}|Q|$ such that
for every $x\in E$
\[
\omega_{\frac{1}{2^{n+2}}}(b,Q)\leq|b(x)-m_{b}(Q)|
\]
where $m_{b}(Q)$ is a not necessarily unique number that satisfies
\[
\max\left\{ |\{x\in Q\,:\,b(x)>m_{b}(Q)\}|,\,|\{x\in Q\,:\,b(x)<m_{b}(Q)\}|\right\} \leq\frac{|Q|}{2}.
\]

Now let $A\subset Q$ with $|A|=\frac{1}{2}|Q|$ and such that $b(x)\geq m_{b}(Q)$
for every $x\in A$. 
We call $B=Q\setminus A$. Then $|B|=\frac{1}{2}|Q|$
and $b(x)\leq m_{b}(Q)$ for every $x\in B$. 

As $Q$ is the disjoint union of $A$ and $B$, at least half of the set $E$ is contained either in $A$ or in $B$. We may assume, without loss of generality, that half of $E$ is in $A$, so we have
\[
|E\cap A|=|E|-|E\cap (E\cap A)^c|\geq |E|-\frac{|E|}{2}=\frac{1}{2^{n+3}}|Q|.
\]
We also have then that 
\[
|B\cap(E\cap A)^{c}|=|B|-|B\cap(E\cap A)|\geq\frac{1}{2}|Q|-\frac{1}{2^{n+3}}|Q|=\left(\frac{1}{2}-\frac{1}{2^{n+3}}\right)|Q|
\]
So choosing 
\[
A'=A\cap E\qquad B'=B\cap(E\cap A)^{c}
\]
we have that if $y\in A'$ and $x\in B'$ then $\omega_{\frac{1}{2^{n+2}}}(b,Q)\leq b(y)-m_{b}(Q)\leq b(y)-b(x)$.
Consequently, using H\"older's inequality and the hypothesis on $(I_\alpha)_b^m$,
\[
\begin{split}\omega_{\frac{1}{2^{n+2}}}(b,Q)^{m}|A'||B'| & \leq\int_{A'}\int_{B'}(b(y)-b(x))^{m}dxdy\\
 & \leq \ell(Q)^{n-\alpha}\int_{A'}\int_{B'}\frac{(b(y)-b(x))^{m}}{|x-y|^{n-\alpha}}dxdy\\
 & =\ell(Q)^{n-\alpha}\int_{A'}(I_{\alpha})_{b}^{m}(\chi_{B'})(x)dx\\
 & \leq \ell(Q)^{n-\alpha}\left(\int_{Q}\lambda^{-q'}\right)^{\frac{1}{q'}}\left(\int_{\mathbb{R}^{n}}(I_{\alpha})_{b}^{m}(\chi_{B'})(x)^{q}\lambda(x)^{q}dx\right)^{\frac{1}{q}}\\
 & \leq c\ell(Q)^{n-\alpha}\left(\int_{Q}\lambda^{-q'}\right)^{\frac{1}{q'}}\left(\int_{Q}\mu^{p}\right)^{\frac{1}{p}}\\
 & =c|Q|^2\left(\frac{1}{|Q|}\int_{Q}\lambda^{-q'}\right)^{\frac{1}{q'}}\left(\frac{1}{|Q|}\int_{Q}\mu^{p}\right)^{\frac{1}{p}},
\end{split}
\]
where we used that 
\[
\frac{1}{q}+\frac{\alpha}{n}=\frac{1}{p}\iff\frac{1}{q'}+\frac{1}{p}=1+\frac{\alpha}{n}.
\]
And this yields
\[
\begin{split}\omega_{\frac{1}{2^{n+2}}}(b,Q)^{m} & \leq c\left(\frac{1}{|Q|}\int_{Q}\lambda^{-q'}\right)^{\frac{1}{q'}}\left(\frac{1}{|Q|}\int_{Q}\mu^{p}\right)^{\frac{1}{p}}.
\end{split}
\]
Since $\mu\in A_{p,q}$ we have that $\mu^p\in A_p
$. Hence (see for example \cite{DMRO})
  \[\frac{1}{|Q|}\int_Q\mu^p\leq c\left(\frac{1}{|Q|}\int_Q \mu^{\frac{p}{r}}\right)^r,\qquad \text{for any }r>1.\]
Using that inequality for some $r>1$ to be chosen combined with H\"older's inequality with $\beta=\frac{r}{p}\frac{1}{m}$, 
we have that
\[
\begin{split}\left(\frac{1}{|Q|}\int_{Q}\mu^{p}\right)^{\frac{1}{p}} & \leq c\left(\frac{1}{|Q|}\int_{Q}\mu^{\frac{p}{r}}\lambda^{-\frac{p}{r}}\lambda^{\frac{p}{r}}\right)^{\frac{r}{p}}\leq c\left(\frac{1}{|Q|}\int_{Q}\nu^{\frac{1}{m}}\right)^{m}\left(\frac{1}{|Q|}\int_{Q}\lambda^{\frac{p}{r}\beta'}\right)^{\frac{r}{p\beta'}}\\
 & =c\left(\frac{1}{|Q|}\int_{Q}\nu^{\frac{1}{m}}\right)^{m}\left(\frac{1}{|Q|}\int_{Q}\lambda^{\frac{p}{r-pm}}\right)^{\frac{r-pm}{p}}
\end{split}
\]
and choosing $r=pm+1$ we obtain
\[
\left(\frac{1}{|Q|}\int_{Q}\mu^{p}\right)^{\frac{1}{p}}\leq c\left(\frac{1}{|Q|}\int_{Q}\nu^{\frac{1}{m}}\right)^{m}\left(\frac{1}{|Q|}\int_{Q}\lambda^{p}\right)^{\frac{1}{p}}.
\]
This yields
\[
\omega_{\frac{1}{2^{n+2}}}(b,Q)^{m}\leq c\left(\frac{1}{|Q|}\int_{Q}\nu^{\frac{1}{m}}\right)^{m}\left(\frac{1}{|Q|}\int_{Q}\lambda^{-q'}\right)^{\frac{1}{q'}}\left(\frac{1}{|Q|}\int_{Q}\lambda^{p}\right)^{\frac{1}{p}}.
\]
Now we observe that since $q>p$ then by H\"older's inequality
\[
\left(\frac{1}{|Q|}\int_{Q}\lambda^{p}\right)^{\frac{1}{p}}\leq\left(\frac{1}{|Q|}\int_{Q}\lambda^{q}\right)^{\frac{1}{q}}
\quad\text{and}\quad
\left(\frac{1}{|Q|}\int_{Q}\lambda^{-q'}\right)^{\frac{1}{q'}}\leq\left(\frac{1}{|Q|}\int_{Q}\lambda^{-p'}\right)^{\frac{1}{p'}}.
\]
Thus
\[
\left(\frac{1}{|Q|}\int_{Q}\lambda^{-q'}\right)^{\frac{1}{q'}}\left(\frac{1}{|Q|}\int_{Q}\lambda^{p}\right)^{\frac{1}{p}}\leq\left[\left(\frac{1}{|Q|}\int_{Q}\lambda^{q}\right)\left(\frac{1}{|Q|}\int_{Q}\lambda^{-p'}\right)^{\frac{q}{p'}}\right]^{\frac{1}{q}}.
\]
Consequently, since $\lambda\in A_{p,q}$, we finally get
\[
\omega_{\frac{1}{2^{n+2}}}(b,Q)\leq c\frac{1}{|Q|}\int_{Q}\nu^{\frac{1}{m}}
\]
and we are done in view of Lemma \ref{BMO}.
\subsection{Proof of Corollary \ref{cor:oneweight}}
To prove the corollary it suffices to estimate each term in $\kappa_m$ in Theorem \ref{thm:BloomFrac} for $\mu=\lambda$. Indeed, we observe first that taking $\mu=\lambda$
\[\kappa_m=[\mu]^{\left(1-\frac{\alpha}{n}\right)\max\left\{1,\frac{p'}{q}\right\}}_{A_{p,q}} \sum_{h=0}^m\binom{m}{h}[\mu^q]_{A_q}^{(m-h)\max\{1,\frac{1}{q-1}\}}[\mu^p]_{A_p}^{h\max\{1,\frac{1}{p-1}\}}.
\]
Now, taking into account \eqref{eq:ApAqApq}, we get
\[
[\mu^{q}]_{A_{q}}^{(m-h)\max\left\{ 1,\frac{1}{q-1}\right\} }\leq[\mu]_{A_{p,q}}^{(m-h)\max\left\{ 1,\frac{1}{q-1}\right\} }
\]
and 
\[
[\mu^{p}]_{A_{p}}^{h\max\left\{ 1,\frac{1}{p-1}\right\} }\leq[\mu]_{A_{p,q}}^{h\frac{p}{q}\max\left\{ 1,\frac{1}{p-1}\right\} }.
\]
Consequently
\[
\kappa_m\leq c_m [\mu]_{A_{p,q}}^{\left(1-\frac{\alpha}{n}\right)\max\left\{ 1,\frac{p'}{q}\right\} +m\max\left\{ 1,\frac{1}{q-1},\frac{p}{q},\frac{p'}{q}\right\}}.
\]
Note that since $p<q$ we have that $\frac{p}{q}<1$ and also
\[
p<q\iff p'>q'\iff\frac{1}{q-1}<\frac{p'}{q}.
\]
This yields that $\max\left\{ 1,\frac{1}{q-1},\frac{p}{q},\frac{p'}{q}\right\} =\max\left\{ 1,\frac{p'}{q}\right\}$
and we have that 
\[
\kappa_m \leq c_m[\mu]_{A_{p,q}}^{\left(m+1-\frac{\alpha}{n}\right)\max\left\{ 1,\frac{p'}{q}\right\} },
\]
as we wanted to prove.

To establish the sharpness of the exponent in \eqref{Cor1} we will use the adaption of Buckley's example \cite{Buck} to the fractional setting that was devised in \cite{CUM}. First we observe that we can  restrict ourselves to the case in which $p'/q\geq 1$, since the case $p'/q<1$ follows at once by duality, taking into account the fact that $(I_\alpha)_b^m$ is essentially self-adjoint (in this case, $[(I_\alpha)_b^m]^*=(-1)^m(I_\alpha)_b^m$) and the fact that if $w\in A_{p,q}$, then $w^{-1}\in A_{q',p'}$ and $[w^{-1}]_{A_{q',p'}}=[w]_{A_{p,q}}^{p'/q}$.

Suppose then that $p'/q\geq 1$, and take $\delta\in (0,1)$. Define the weight $w_\delta(x)=|x|^{(n-\delta)/p'}$ and the power functions $f_\delta(x)=|x|^{\delta-n}\chi_{B(0,1)}(x)$. Easy computations yield
\[ \|f_\delta\|_{L^{p}(w_\delta^p)}\asymp \delta^{-1/p},\quad \text{and}\quad [w_\delta]_{A_{p,q}}\asymp \delta^{-q/p'}.\]
Let $b$ be the $\BMO$ function $b(x)=\log|x|$. For $x\in \mathbb{R}^n$ with $|x|\geq 2$, we have that 
\[
\begin{split}
(I_\alpha)_b^mf_\delta(x)&=\int_{B(0,1)}\frac{\log^m(|x|/|y|)}{|x-y|^{n-\alpha}}|y|^{\delta-n}dy\\
&\geq|x|^{\delta-n+\alpha}\int_{B(0,|x|^{-1})}\frac{\log^m(1/|z|)}{(1+|z|)^{n-\alpha}}|z|^{\delta-n}dz\\
&\geq c_n \frac{|x|^\delta}{(1+|x|)^{n-\alpha}}\int_0^{|x|^{-1}}\log^m(1/r)r^{\delta-1}dr.
\end{split}
\]
Integration by parts yields
\[
\begin{split}
  \int_0^{|x|^{-1}}\log^m(1/r)r^{\delta-1}dr&=\delta^{-1}|x|^{-\delta}\log^m|x|\sum_{k=0}^m\frac{m!}{(m-k)!\delta^k\log^k|x|}\\
  &\geq\delta^{-1}|x|^{-\delta}\log^m|x| \sum_{k=0}^m\binom{m}{k}(\delta^{-1}\log^{-1}|x|)^k\\
  &=\delta^{-1}|x|^{-\delta}\log^m|x|(\delta^{-1}\log^{-1}|x|+1)^m\\
  &\geq \delta^{-m-1}|x|^{-\delta}.
\end{split}
\]
Then,
\[
(I_\alpha)_b^mf_\delta(x)\geq \frac{c_n}{\delta^{m+1}|x|^{n-\alpha}},\qquad |x|\geq2.
\]
Hence, taking into account that $\frac{1}{q}+\frac{\alpha}{n}=\frac{1}{p}$, we have that
\[
\begin{split}
  \|(I_\alpha)_b^mf_\delta\|_{L^q(w_\delta^q)}&\geq c_n\delta^{-(m+1)}\left(\int_{|x|\geq 2}\frac{|x|^{(n-\delta)\frac{q}{p'}}}{|x|^{(n-\alpha)q}}dx\right)^{1/q}\\
  &=c_n\delta^{-(m+1)}\left(\int_{|x|\geq 2}|x|^{-\delta q/p'-n}dx\right)^{1/q}\\
  &=c\delta^{-(m+1)-\frac{1}{q}}\\
  &=c[w_\delta]_{A_{p,q}}^{\left(m+1-\frac{\alpha}{n}\right)\frac{p'}{q}}\|f_\delta\|_{L^p(w_\delta^p)}\\
  &=c[w_\delta]_{A_{p,q}}^{\left(m+1-\frac{\alpha}{n}\right)\max\left\{1,\frac{p'}{q}\right\}}\|f_\delta\|_{L^p(w_\delta^p)},
\end{split}
\]
so the exponent in \eqref{Cor1} is sharp.

\begin{rem}
We would like to point out that an alternative argument to settle the sharpness we have just obtained follows from the combination of arguments in Sections 3.1 and 3.4 in \cite{LPR}.
\end{rem}

\subsection{Some further remarks. Mixed $A_{p,q}-A_{\infty}$ bounds}
We recall that  $w\in A_\infty$ if and only if
\[[w]_{A_\infty}=\sup_Q\frac{1}{w(Q)}\int_QM(w\chi_Q)<\infty\]
and that is, up until now (see \cite{FH}), the smallest constant characterizing the $A_\infty$ class. 
We would like to point out that it would be possible to provide mixed estimates for $(I_\alpha)_b^m$ in terms of this $A_\infty$ constant. For that purpose it suffices to follow the same argument used to establish Theorem \ref{thm:BloomFrac}, but taking into account that, if $w\in A_p$ and we call $\sigma=w^{\frac{1}{1-p}}$, then
\[\|A_\mathcal{S}\|_{L^p(w)}\leq c_{n,p}[w]_{A_p}^{\frac{1}{p}}\left([w]_{A_\infty}^{\frac{1}{p'}}+[\sigma]_{A_\infty}^{\frac{1}{p}}\right).\]
Also, if $\alpha\in(0,n)$,  $1<p<\frac{n}{\alpha}$, $q$ is defined by $\frac{1}{q}+\frac{\alpha}{n}=\frac{1}{p}$ and $w\in A_{p,q}$ then, if, again, $\sigma=w^{\frac{1}{1-p}}$,
\[\|I_\mathcal{S}^\alpha\|_{L^p(w^p)\rightarrow L^q(w^q)}\leq c_{n,p}[w]_{A_{p,q}}\left([w^q]_{A_\infty}^{\frac{1}{p'}}+[\sigma^p]_{A_\infty}^{\frac{1}{q}}\right).\]
The preceding estimate was established in \cite{CUM2} and is also contained in the recent work \cite{FH}.

\section*{Acknowledgements}
The third author would like to thank D. Cruz-Uribe for drawing his attention to the problem of sparse domination for iterated commutators of fractional integrals during the Spring School on Analysis 2017 held at Paseky nad Jizerou.

The first and the third authors are supported by the Spanish Ministry of E\-co\-no\-my and Competitiveness grant though the project MTM2014-53850-P and also by the grant IT-641-13 of the Basque Government.

The second and the third authors are supported by the Basque Government through the BERC 2014-2017 program and by Spanish Ministry of Economy and Competitiveness MINECO through BCAM Severo Ochoa excellence accreditation SEV-2013-0323.

The second author is also supported by ``la Caixa'' grant.

\printbibliography

\end{document}